\documentclass[12pt,leqno]{article}
\usepackage{hyperref}
\usepackage{soul}
\usepackage[doc]{optional}
\usepackage{xcolor}
\definecolor{labelkey}{rgb}{0,0.08,0.45}
\definecolor{refkey}{rgb}{0,0.6,0.0}
\definecolor{Brown}{rgb}{0.45,0.0,0.05}
\usepackage{exscale,relsize}
\usepackage{amsmath}
\usepackage{amsfonts}
\usepackage{amssymb}
\usepackage{calc}
\usepackage{theorem}
\usepackage{pifont}      
\usepackage{graphicx}
\DeclareMathOperator{\weakstarly}{\stackrel{\mathrm{w*}}{\rightharpoondown}}

\newcommand{\scal}[2]{\langle{{#1},{#2}}\rangle}

\newcommand{\RR}{\ensuremath{\mathbb R}}

\newcommand{\RX}{\ensuremath{\,\left]-\infty,+\infty\right]}}
\newcommand{\RXX}{\ensuremath{\,\left[-\infty,+\infty\right]}}

\newcommand{\NN}{\ensuremath{\mathbb N}}

\newcommand{\menge}[2]{\big\{{#1} \mid {#2}\big\}}

\newcommand{\To}{\ensuremath{\rightrightarrows}}

\newcommand{\di}{\ensuremath{\operatorname{dist}}}

\newcommand{\dom}{\ensuremath{\operatorname{dom}}}

\newcommand{\gra}{\ensuremath{\operatorname{gra}}}

\newcommand{\intdom}{\ensuremath{\operatorname{int}\operatorname{dom}}\,}
\newcommand{\inte}{\ensuremath{\operatorname{int}}}

\renewcommand{\phi}{\ensuremath{\varphi}}

\newtheorem{theorem}{Theorem}[section]

\newtheorem{fact}[theorem]{Fact}
\newtheorem{corollary}[theorem]{Corollary}
\newtheorem{proposition}[theorem]{Proposition}

\theoremstyle{plain}{\theorembodyfont{\rmfamily}
}
\theoremstyle{plain}{\theorembodyfont{\rmfamily}
}
\theoremstyle{plain}{\theorembodyfont{\rmfamily}
}
\theoremstyle{plain}{\theorembodyfont{\rmfamily}
\newtheorem{example}[theorem]{Example}}
\theoremstyle{plain}{\theorembodyfont{\rmfamily}
\newtheorem{remark}[theorem]{Remark}}
\theoremstyle{plain}{\theorembodyfont{\rmfamily}
}

\begin{document}


\title{\sffamily{
The sum of a maximal monotone
 operator  of type (FPV) and a maximal monotone operator with full domain\\ is maximal monotone}}

\author{
Liangjin\
Yao\thanks{Mathematics, Irving K.\ Barber School, UBC Okanagan,
Kelowna, British Columbia V1V 1V7, Canada.
E-mail:  \texttt{ljinyao@interchange.ubc.ca}.}}
 \vskip 3mm

\date{August 13, 2010}
\maketitle

\begin{abstract} \noindent
The most important open problem in Monotone Operator Theory
concerns the maximal monotonicity of the sum of two
maximal monotone operators provided that
Rockafellar's constraint qualification holds.

In this paper, we prove the maximal monotonicity of $A+B$
provided that
$A$ and $B$ are maximal monotone operators such that
$\dom A\cap\inte\dom B\neq\varnothing$,
$A+N_{\overline{\dom B}}$ is of type (FPV), and
$\dom A\cap\overline{\dom B}\subseteq\dom B$.
The proof utilizes the
Fitzpatrick function in an essential way.
\end{abstract}



\section{Introduction}

Throughout this paper, we assume that
$X$ is a real Banach space with norm $\|\cdot\|$,
that $X^*$ is the continuous dual of $X$, and
that $X$ and $X^*$ are paired by $\scal{\cdot}{\cdot}$.
Let $A\colon X\To X^*$
be a \emph{set-valued operator} (also known as multifunction)
from $X$ to $X^*$, i.e., for every $x\in X$, $Ax\subseteq X^*$,
and let
$\gra A = \menge{(x,x^*)\in X\times X^*}{x^*\in Ax}$ be
the \emph{graph} of $A$.
Recall that $A$ is  \emph{monotone} if
\begin{equation}
\scal{x-y}{x^*-y^*}\geq 0,\quad \forall (x,x^*)\in \gra A\;\forall (y,y^*)\in\gra A,
\end{equation}
and \emph{maximal monotone} if $A$ is monotone and $A$ has no proper monotone extension
(in the sense of graph inclusion).
Let $A:X\rightrightarrows X^*$ be monotone and $(x,x^*)\in X\times X^*$.
 We say $(x,x^*)$ is \emph{monotonically related to}
$\gra A$ if
\begin{align*}
\langle x-y,y-y^*\rangle\geq0,\quad \forall (y,y^*)\in\gra A.\end{align*}
Let $A:X\rightrightarrows X^*$ be maximal monotone. We say $A$ is
\emph{of type (FPV)} if  for every open convex set $U\subseteq X$ such that
$U\cap \dom A\neq\varnothing$, the implication
\begin{equation*}
x\in U\text{and}\,(x,x^*)\,\text{is monotonically related to $\gra A\cap U\times X^*$}
\Rightarrow (x,x^*)\in\gra A
\end{equation*}
holds.
We say $A$ is a \emph{linear relation} if $\gra A$ is a linear subspace.
Monotone operators have proven to be a key class of objects
in modern Optimization and Analysis; see, e.g.,
the books
\cite{BorVan,BurIus,ButIus,ph,Si,Si2,RockWets,Zalinescu}
and the references therein.
We adopt standard notation used in these books:
$\dom A= \menge{x\in X}{Ax\neq\varnothing}$ is the \emph{domain} of $A$.
Given a subset $C$ of $X$,
$\inte C$ is the \emph{interior} of $C$, and
$\overline{C}$ is the norm \emph{closure} of $C$.
The \emph{indicator function} of $C$, written as $\iota_C$, is defined
at $x\in X$ by
\begin{align}
\iota_C (x)=\begin{cases}0,\,&\text{if $x\in C$;}\\
\infty,\,&\text{otherwise}.\end{cases}\end{align}
We set $\di(x, C)=\inf_{c\in C}\|x-c\|$, for $x\in X$.
If $D\subseteq X$, we set $C-D=\{x-y\mid x\in C, y\in D\}$.
  For every $x\in X$, the normal cone operator of $C$ at $x$
is defined by $N_C(x)= \menge{x^*\in
X^*}{\sup_{c\in C}\scal{c-x}{x^*}\leq 0}$, if $x\in C$; and $N_C(x)=\varnothing$,
if $x\notin C$.
For $x,y\in X$, we set $\left[x,y\right]=\{tx+(1-t)y\mid 0\leq t\leq 1\}$.
 Given $f\colon X\to \RX$, we set
$\dom f= f^{-1}(\RR)$ and
$f^*\colon X^*\to\RXX\colon x^*\mapsto
\sup_{x\in X}(\scal{x}{x^*}-f(x))$ is
the \emph{Fenchel conjugate} of $f$.
If $f$ is convex and $\dom f\neq\varnothing$, then
   $\partial f\colon X\To X^*\colon
   x\mapsto \menge{x^*\in X^*}{(\forall y\in
X)\; \scal{y-x}{x^*} + f(x)\leq f(y)}$
is the \emph{subdifferential operator} of $f$.
We also set $P_X: X\times X^*\rightarrow X\colon (x,x^*)\mapsto x$.
Finally,  the \emph{open unit ball} in $X$ is denoted by
$\mathbb{B}_X= \menge{x\in X}{\|x\|< 1}$, and $\NN=\{1,2,3,\ldots\}$.

Let $A$ and $B$ be maximal monotone operators from $X$ to
$X^*$.
Clearly, the \emph{sum operator} $A+B\colon X\To X^*\colon x\mapsto
Ax+Bx = \menge{a^*+b^*}{a^*\in Ax\;\text{and}\;b^*\in Bx}$
is monotone.
Rockafellar's \cite[Theorem~1]{Rock70} guarantees maximal monotonicity
of $A+B$ under
\emph{Rockafellar's constraint qualification}
$\dom A \cap\intdom B\neq \varnothing$ when $X$ is reflexive
--- this result is often referred to as ``the sum theorem''.
The most famous open problem concerns the maximal monotonicity of $A+B$ in
nonreflexive Banach spaces when Rockafellar's constraint qualification
holds.
See Simons' monograph
\cite{Si2} and \cite{Bor1, Bor2, ZalVoi}
for a comprehensive account of some recent developments.

Now we focus on the  case when $A$ and $B$
satisfy the following three conditions:
$\dom A\cap\inte\dom B\neq\varnothing$,
$A+N_{\overline{\dom B}}$ is of type (FPV),
and $\dom A\cap\overline{\dom B}\subseteq\dom B$.
 We show that the sum $A+B$ is maximal monotone in this setting.
We note in passing that in \cite[Corollary~2.9(a)]{VV},
Verona and Verona derived the same conclusion when $A$
 is the subdifferential operator of a proper lower semicontinuous convex function,
and $ B$ is  maximal monotone
with full domain.
In {\cite[Theorem~3.1]{BWY4}}, it was recently shown that
the sum theorem is true when $A$ is a linear relation and $B$ is the normal
cone operator of a closed convex set. In \cite{Voisei09},
Voisei confirmed  \cite[Theorem~41.5]{Si} that  the sum theorem is also true when $A$ is type of  (FPV)  with convex domain,
and $B$ is the normal
cone operator of a closed convex set.
Our main result, Theorem~\ref{t:main},
generalizes all the above results and it also contains a result due to
Heisler \cite[Remark, page~17]{ph2}
on the sum theorem for two operators with full domain.

The remainder of this paper is organized as follows.
In Section~\ref{s:aux}, we collect auxiliary results for future reference
and for the
reader's convenience.
The main result (Theorem~\ref{t:main}) is proved
in Section~\ref{s:main}.

\section{Auxiliary Results}
\label{s:aux}

\begin{fact}[Rockafellar] \label{f:F4}
\emph{(See {\cite[Theorem~3(b)]{Rock66}},
{\cite[Theorem~18.1]{Si2}}, or
{\cite[Theorem~2.8.7(iii)]{Zalinescu}}.)}\\
Let $f,g: X\rightarrow\RX$ be proper convex functions.
Assume that there exists a point $x_0\in\dom f \cap \dom g$
such that $g$ is continuous at $x_0$.
Then  $\partial (f+g)=\partial f+\partial g$.
\end{fact}

\begin{fact}\emph{(See \cite[Theorem~2.28]{ph}.)}
\label{pheps:11}Let $A:X\To X^*$ be  monotone with $\inte\dom A\neq\varnothing$.
Then $A$ is locally bounded at $x\in\inte\dom A$, i.e., there exist $\delta>0$ and $K>0$ such that
\begin{align*}\sup_{y^*\in Ay}\|y^*\|\leq K,\quad \forall y\in (x+\delta \mathbb{B_X})\cap \dom A.
\end{align*}
\end{fact}

\begin{fact}[Fitzpatrick]
\emph{(See {\cite[Corollary~3.9]{Fitz88}}.)}
\label{f:Fitz}
Let $A\colon X\To X^*$ be maximal monotone,  and set
\begin{equation}
F_A\colon X\times X^*\to\RX\colon
(x,x^*)\mapsto \sup_{(a,a^*)\in\gra A}
\big(\scal{x}{a^*}+\scal{a}{x^*}-\scal{a}{a^*}\big),
\end{equation}
 the \emph{Fitzpatrick function} associated with $A$.
Then for every $(x,x^*)\in X\times X^*$, the inequality
$\scal{x}{x^*}\leq F_A(x,x^*)$ is true,
and the equality holds if and only if $(x,x^*)\in\gra A$.
\end{fact}

\begin{fact}
\emph{(See \cite[Theorem~3.4 and Corollary~5.6]{Voi1}, or \cite[Theorem~24.1(b)]{Si2}.)}
\label{f:referee1}
Let $A, B:X\To X^*$ be maximal monotone operators. Assume
$\bigcup_{\lambda>0} \lambda\left[P_X(\dom F_A)-P_X(\dom F_B)\right]$
is a closed subspace.
If
\begin{equation}
F_{A+B}\geq\langle \cdot,\,\cdot\rangle\;\text{on \; $X\times X^*$},
\end{equation}
then $A+B$ is maximal monotone.
\end{fact}

\begin{fact}[Simons]
\emph{(See \cite[Thereom~27.1 and Thereom~27.3]{Si2}.)}
\label{f:referee02c}
Let $A:X\To X^*$ be  maximal monotone with $\inte\dom A\neq\varnothing$. Then
$\inte\dom A=\inte\left[P_X\dom F_A\right]$,
$\overline{\dom A}=\overline{P_X\left[\dom F_A\right]}$ and $\overline{\dom A}$ is convex.
\end{fact}

Now we cite some results on maximal monotone  operators of type (FPV).
\begin{fact}[Simons]
\emph{(See \cite[Theorem~48.4(d)]{Si2}.)}
\label{f:referee0d}
Let $f:X\rightarrow\RX$ be proper, lower semicontinuous, and convex.
Then $\partial f$ is of type (FPV).
\end{fact}

\begin{fact}[Simons]
\emph{(See \cite[Theorem~46.1]{Si2}.)}
\label{f:referee01}
Let $A:X\To X^*$ be a maximal monotone linear relation.
Then $A$ is of type (FPV).
\end{fact}

\begin{fact}[Simons and Verona-Verona]
\emph{(See \cite[Thereom~44.1]{Si2} or \cite{VV1}.)}
\label{f:referee02a}
Let $A:X\To X^*$ be a maximal monotone. Suppose that
for every closed convex subset $C$ of $X$
with $\dom A \cap \inte C\neq \varnothing$, the operator
$A+N_C$ is maximal monotone.
Then $A$ is of type  (FPV).
\end{fact}

The following statement first appeared in \cite[Theorem~41.5]{Si}.
However, on \cite[page~199]{Si2}, concerns were raised about the validity
of the proof of \cite[Theorem~41.5]{Si}.
In \cite{Voisei09},
Voisei recently provided a result that generalizes
and confirms \cite[Theorem~41.5]{Si} and hence the following fact.

\begin{fact}[Voisei]
\label{ft:maina}
Let $A:X\To X^*$ be  maximal monotone of type  (FPV)  with convex domain,
let $C$ be a nonempty closed convex subset of $X$,
and suppose that $\dom A \cap \inte C\neq \varnothing$.
Then $A+N_C$ is maximal monotone.
\end{fact}

\begin{corollary}\label{domain:L1}
Let $A:X\To X^*$ be  maximal monotone of type  (FPV)  with convex domain,
let $C$ be a nonempty closed convex subset of $X$,
and suppose that $\dom A \cap \inte C\neq \varnothing$.
Then $A+N_C$ is of type  $(FPV)$.

\end{corollary}
\begin{proof} By Fact~\ref{ft:maina}, $A+N_C$ is maximal monotone.
Let $D$ be a nonempty closed convex subset of $X$,
and suppose that $\dom (A+N_C) \cap \inte D\neq \varnothing$.
Let $x_1\in \dom A \cap \inte C$ and $x_2\in \dom (A+N_C) \cap \inte D$.
Thus, there exists   $\delta>0$ such that $x_1+\delta \mathbb{B}_X\subseteq C$
 and $x_2+\delta \mathbb{B}_X\subseteq D$.
Then for small enough $\lambda\in\left]0,1\right[$, we have
 $x_2+\lambda (x_1-x_2)+\tfrac{1}{2}\delta \mathbb{B}_X\subseteq D$.
Clearly, $x_2+\lambda(x_1-x_2)+\lambda\delta \mathbb{B}_X\subseteq C$.
Thus $x_2+\lambda(x_1-x_2)+\tfrac{\lambda\delta}{2} \mathbb{B}_X\subseteq C\cap D$.
Since $\dom A$ is convex, $x_2+\lambda(x_1-x_2)\in\dom A$
 and  $x_2+\lambda(x_1-x_2)\in\dom A\cap\inte(C\cap D)$.
By Fact~\ref{f:F4} , $A+N_C+N_D=A+N_{C\cap D}$.
Then, by Fact~\ref{ft:maina} (applied to $A$ and $C\cap D$),
$A+N_C+N_D=A+N_{C\cap D}$ is maximal monotone.
By Fact~\ref{f:referee02a},   $A+N_C$ is of type  $(FPV)$.
\end{proof}

\begin{corollary}\label{domain:L2}
Let $A:X\To X^*$ be a maximal monotone linear relation,
let $C$ be a nonempty closed convex subset of $X$,
and suppose that $\dom A \cap \inte C\neq \varnothing$.
Then $A+N_C$ is of type  $(FPV)$.

\end{corollary}

\begin{proof}
Apply Fact~\ref{f:referee01} and Corollary~\ref{domain:L1}.
\end{proof}

\section{Main Result}
\label{s:main}

The following result plays a key role in the proof of Theorem~\ref{t:main}.
The first half of its proof  follows along the lines of the proof of
\cite[Theorem~44.2]{Si2}.

\begin{proposition}\label{PGV:1}
Let $A, B:X\To X^*$ be  maximal monotone
with $\dom A\cap\inte\dom B\neq\varnothing$. Assume
that $A+N_{\overline{\dom B}}$ is maximal monotone of type (FPV),
and
$\dom A\cap\overline{\dom B}\subseteq\dom B$. Then
$\overline{P_X\left[\dom F_{A+B}\right]}=\overline{\dom A\cap\dom B}$.
\end{proposition}
\begin{proof}
By \cite[Theorem~3.4]{Fitz88}, $\overline{\dom A\cap\dom B}=\overline{\dom (A+B)}
\subseteq \overline{P_X\left[\dom F_{A+B}\right]}$.
It suffices to show that
\begin{align}
P_X\left[\dom F_{A+B}\right]\subseteq\overline{\dom A\cap\dom B}.\label{FCG:10}
\end{align}
After translating the graphs if necessary, we can and do
assume that $0\in\dom A\cap\inte\dom B$ and that $(0,0)\in\gra B$.
\opt{com}{(We can do it in this way: Assume that $a\in\dom A\cap\inte\dom B$
 and $(a,b^*)\in\gra B$.
Let $A':X\rightrightarrows X^*$ and $B':X\rightrightarrows X^*$ be such that
\begin{align}
\gra A'=\gra A-\{(a,0)\},\quad \gra B'=\gra B-\{(a,b^*)\}.\label{Revt:1}\end{align}
Then we have $A', B'$ are maximal monotone with $0\in\dom A'\cap\inte\dom B'$ and $(0,0)\in\gra B'$.
Since  $(A'+N_{\overline{\dom B'}})(\cdot)=A'(\cdot)+N_{\overline{\dom B}-\{a\}}(\cdot)
=A(\cdot+a)+N_{\overline{\dom B}}(\cdot+a)=(A+N_{\overline{\dom B}})(\cdot+a)$, by the assumption,
$A'+N_{\overline{\dom B'}}$ is  maximal monotone of type (FPV).
By the assumption, we also have that
$\dom A'\cap\overline{\dom B'}=\left[\dom A-\{a\}\right]\cap\left[\overline{\dom B}-\{a\}\right]
\subseteq\left[{\dom B}-\{a\}\right]=\dom B'$.
By \eqref{Revt:1}, $(A'+B')(\cdot)=(A+B)(\cdot+a)-\{b^*\}$.
Clearly,
$A+B$ is maximal monotone if and only if $A'+B'$ is maximal monotone.)}

To show \eqref{FCG:10}, we take $z\in P_X\left[\dom F_{A+B}\right]$ and
we assume to the contrary that
\begin{align}z\notin\overline{\dom A\cap\dom B}\label{FPCGG:1}.\end{align}
Thus $\alpha= \di(z, \overline{\dom A\cap\dom B})>0$.
Now take $y^*_0\in X^*$ such that
\begin{align}
\|y^*_0\|=1\quad\text{and}\quad \langle z,y^*_0\rangle\geq\tfrac{2}{3}\|z\|.\label{FG:1}
\end{align}
Set
\begin{align}
U_n= \left[0,z\right]+\tfrac{\alpha}{4n}\mathbb{B}_X,\quad \forall n\in\NN\label{FDD:1}.
\end{align}
Since  $0\in N_{\overline{\dom B}}(x), \forall x\in\dom B$,
$\gra B\subseteq \gra (B+N_{\overline{\dom B}})$.
Since $B$ is maximal monotone and $B+N_{\overline{\dom B}}$ is a monotone
extension of $B$, we must have $B=B+N_{\overline{\dom B}}$.
Thus
\begin{align}
A+B=A+N_{\overline{\dom B}}+B.\label{FG:14}
\end{align}
Since $\dom A\cap\overline{\dom B}\subseteq\dom B$ by assumption,
we obtain
\begin{align*}
\dom A\cap\dom B\subseteq\dom (A+N_{\overline{\dom B}})=\dom A\cap\overline{\dom B}\subseteq \dom A\cap\dom B.
\end{align*}
Hence
\begin{align}
\dom A\cap\dom B=\dom (A+N_{\overline{\dom B}}).\label{FCGG:u1}
\end{align}
By \eqref{FPCGG:1} and \eqref{FCGG:u1}, $z\notin\dom (A+N_{\overline{\dom B}})$
and thus $(z,ny^*_0)\notin\gra (A+N_{\overline{\dom B}}), \forall n\in\NN$.
For every $n\in\NN$, since $z\in U_n$
and since $A+N_{\overline{\dom B}}$ is of type (FPV) by assumption,
we deduce the existence of
 $(z_n,z^*_n)\in\gra  (A+N_{\overline{\dom B}})$ such that $z_n\in U_n$ and
\begin{align}
\langle z-z_n, z^*_n\rangle>n\langle z-z_n,y^*_0\rangle,\quad \forall n\in\NN.\label{FG:7}
\end{align}
Hence, using \eqref{FDD:1},
there exists $\lambda_n\in\left[0,1\right]$ such that
\begin{align}
\|z-z_n-\lambda_n z\|=\|z_n-(1-\lambda_n) z\|<\tfrac{1}{4}\alpha,\quad \forall n\in\NN.\label{FG:2}
\end{align}
By the triangle inequality, we have
$\|z-z_n\|<\lambda_n\| z\|+\tfrac{1}{4}\alpha$ for every $n\in\NN$.
From the definition of $\alpha$ and \eqref{FCGG:u1},
it follows that $\alpha\leq\|z-z_n\|$ and hence that
$\alpha<\lambda_n\| z\|+\tfrac{1}{4}\alpha$.
Thus,
\begin{align}
\tfrac{3}{4}\alpha<\lambda_n\| z\|,\quad \forall n\in\NN.\label{FG:4}
\end{align}
By \eqref{FG:2} and \eqref{FG:1},
\begin{align}
\langle z-z_n-\lambda_n z,y^*_0\rangle\geq-\|z_n-(1-\lambda_n) z\|>-\tfrac{1}{4}\alpha,\quad \forall n\in\NN.\label{FG:a3}
\end{align}
By \eqref{FG:a3}, \eqref{FG:1} and \eqref{FG:4},
\begin{align}
\langle z-z_n,y^*_0\rangle>\lambda_n\langle z,y^*_0\rangle-\tfrac{1}{4}\alpha>
\tfrac{2}{3}\tfrac{3}{4}\alpha-\tfrac{1}{4}\alpha=\tfrac{1}{4}\alpha,\quad \forall n\in\NN.\label{FG:3}
\end{align}
Then, by \eqref{FG:7} and \eqref{FG:3},
\begin{align}
 \langle z-z_n, z^*_n\rangle>\tfrac{1}{4}n\alpha,\quad \forall n\in\NN.\label{FG:10}
\end{align}
By \eqref{FDD:1}, there exist $t_n\in\left[0,1\right]$
 and $b_n\in\tfrac{\alpha}{4n}\mathbb{B}_X$ such
that $z_n=t_nz+b_n$. Since $t_n\in\left[0,1\right]$, there exists  a
convergent subsequence of $(t_n)_{n\in\NN}$, which, for convenience, we
still denote by $(t_n)_{n\in\NN}$. Then
$t_n\rightarrow \beta$, where $\beta\in\left[0,1\right]$.
Since $b_n\rightarrow 0$, we have
\begin{align}z_n\rightarrow \beta z.\label{FCG:1}\end{align}
By \eqref{FCGG:u1},
$z_n\in\dom A\cap\dom B$; thus, $\|z_n-z\|\geq\alpha$ and
$\beta\in\left[0,1\right[$.
In view of \eqref{FG:14} and \eqref{FG:10}, we have,
for every $z^*\in X^*$,
\begin{align}
&F_{A+B}(z,z^*)=F_{A+N_{\overline{\dom B}}+B}(z,z^*)\nonumber\\
&\geq\sup_{\{n\in\NN, y^*\in X^*\}}\left[\langle z_n,z^*\rangle+\langle z-z_n,z^*_n\rangle
+\langle z-z_n, y^*\rangle -\iota_{\gra B}(z_n,y^*)\right]\nonumber\\
& \geq\sup_{\{n\in\NN, y^*\in X^*\}}\left[\langle z_n,z^*\rangle+\tfrac{1}{4}n\alpha
+\langle z-z_n, y^*\rangle -\iota_{\gra B}(z_n,y^*)\right].\label{FCG:02}
\end{align}
We now claim that
\begin{align}
F_{A+B}(z,z^*)=\infty.\label{FCGG:5}
\end{align}

We consider two cases.

\emph{Case 1}: $\beta=0$.

By \eqref{FCG:1} and Fact~\ref{pheps:11} (applied to $0\in\inte\dom B$),  there exist $N\in\NN$ and $K>0$ such that
\begin{align}
Bz_n\neq\varnothing\quad\text{and}\quad \sup_{y^*\in Bz_n}\|y^*\|\leq K, \quad\forall n\geq N.\label{Ret:1}
\end{align}

Then, by \eqref{FCG:02},
\begin{align*}
F_{A+B}(z,z^*)&\geq \sup_{\{n\geq N, y^*\in X^*\}}\left[\langle z_n,z^*\rangle+\tfrac{1}{4}n\alpha
+\langle z-z_n, y^*\rangle -\iota_{\gra B}(z_n,y^*) \right]\\
&\geq \sup_{\{n\geq N, y^*\in Bz_n\}}\left[-\| z_n\|\cdot\|z^*\|+\tfrac{1}{4}n\alpha
-\|z-z_n\|\cdot \|y^*\| \right]\\
&\geq \sup_{\{n\geq N\}}\left[-\| z_n\|\cdot\|z^*\|+\tfrac{1}{4}n\alpha
-K\|z-z_n\| \right]\quad\text{(by \eqref{Ret:1})}\\
&=\infty\quad\text{(by \eqref{FCG:1})}.
\end{align*}
Thus \eqref{FCGG:5} holds.

\emph{Case 2}: $\beta\neq0$.

Take $v^*_n\in Bz_n$.
We consider two subcases.

\emph{Subcase 2.1}:  $(v^*_n)_{n\in\NN}$ is bounded.

By \eqref{FCG:02},
\begin{align*}
F_{A+B}(z,z^*)&\geq \sup_{\{n\in\NN\}}\left[\langle z_n,z^*\rangle+\tfrac{1}{4}n\alpha
+\langle z-z_n, v_n^*\rangle \right]\\
&\geq \sup_{\{n\in\NN\}}\left[-\| z_n\|\cdot\|z^*\|+\tfrac{1}{4}n\alpha
-\|z-z_n\|\cdot \|v_n^*\| \right]\\
&=\infty\quad\text{(by \eqref{FCG:1} and the boundedness of $(v^*_n)_{n\in\NN}$)}.
\end{align*}
Hence \eqref{FCGG:5} holds.

\emph{Subcase 2.2}: $(v^*_n)_{n\in\NN}$ is unbounded.

We first show
\begin{align}
\limsup_{n\rightarrow\infty}\,\langle z-z_n, v_n^*\rangle \geq 0.
\label{FCG:2}
\end{align}
Since $(v^*_n)_{n\in\NN}$ is unbounded and
after passing to a subsequence if necessary, we assume that
$\|v^*_n\|\neq 0,\forall n\in\NN$ and that $\|v^*_n\|\rightarrow +\infty$.
By $0\in\inte\dom B$ and Fact~\ref{pheps:11},
there exist $\delta>0$ and $M>0$ such that
\begin{align}
By\neq\varnothing \quad\text{and}\quad\sup_{y^*\in By}\|y^*\|\leq M,\quad \forall y\in \delta
\mathbb{B}_X.\label{RVT:10a}
\end{align}
Then we have
\begin{align}
&\langle z_n-y, v^*_n-y^*\rangle\geq0,\quad \forall y\in \delta \mathbb{B}_X, y^*\in By, n\in\NN\nonumber\\
&\Rightarrow  \langle z_n, v^*_n\rangle-\langle y, v^*_n\rangle
+\langle z_n-y, -y^*\rangle\geq0,\quad\forall y\in \delta \mathbb{B}_X, y^*\in By,  n\in\NN\nonumber\\
&\Rightarrow   \langle z_n, v^*_n\rangle-
\langle y, v^*_n\rangle\geq\langle z_n-y, y^*\rangle,\quad\forall y\in
\delta \mathbb{B}_X, y^*\in By,  n\in\NN\nonumber\\
&\Rightarrow \langle z_n, v^*_n\rangle-
\langle y, v^*_n\rangle\geq -(\|z_n\|+\delta) M, \quad\forall y\in \delta
\mathbb{B}_X,  n\in\NN\quad\text{(by \eqref{RVT:10a})}\nonumber\\
&\Rightarrow \langle z_n, v^*_n\rangle
\geq \langle y, v^*_n\rangle -(\|z_n\|+\delta) M, \quad\forall y\in \delta
\mathbb{B}_X,  n\in\NN\nonumber\\
&\Rightarrow \langle z_n, v^*_n\rangle
\geq \delta\|v^*_n\|-(\|z_n\|+\delta) M, \quad \forall n\in\NN\nonumber\\
&\Rightarrow \langle z_n, \tfrac{v^*_n}{\|v^*_n\|}\rangle
\geq\delta-\tfrac{(\|z_n\|+\delta) M}{\|v^*_n\|}, \quad \forall n\in\NN.\label{FCG:3}
\end{align}
By the Banach-Alaoglu Theorem
(see \cite[Theorem~3.15]{Rudin}), there  exist a weak* convergent \emph{subnet}
$(v^*_\gamma)_{\gamma\in\Gamma}$ of $(v^*_n)_{n\in\NN}$, say
\begin{align}\tfrac{v^*_\gamma}{\|v^*_\gamma\|}\weakstarly  w^*\in X^*.\label{FCGG:9}\end{align}
Using \eqref{FCG:1} and
taking the limit in \eqref{FCG:3} along the subnet, we obtain
\begin{align}
 \langle \beta z, w^*\rangle
\geq  \delta.
\end{align}
Since $\beta> 0$, we have
\begin{align}
 \langle  z, w^*\rangle
\geq  \tfrac{\delta}{\beta}>0.\label{FCG:03}
\end{align}
Now we assume to the contrary that
\begin{align*}
\limsup_{n\rightarrow\infty}\langle z-z_n, v_n^*\rangle
< -\varepsilon,
\end{align*}
for some $\varepsilon>0$.

Then, for all $n$ sufficiently large,
\begin{align*}
\langle z-z_n, v^*_n\rangle< -\tfrac{\varepsilon}{2},
\end{align*}
and so
\begin{align}
\langle z-z_n, \tfrac{v^*_n}{\|v^*_n\|}\rangle< -\tfrac{\varepsilon}{2\|v^*_n\|}\label{Rev:a}.
\end{align}
Then by \eqref{FCG:1} and \eqref{FCGG:9}, taking  the limit in \eqref{Rev:a} along the subnet again, we see that
\begin{align*}
\langle z-\beta z, w^*\rangle\leq 0.
\end{align*}
Since $\beta<1$, we deduce $\langle z, w^*\rangle\leq 0$
which contradicts \eqref{FCG:03}.
Hence \eqref{FCG:2} holds.
By \eqref{FCG:02},
\begin{align*}
F_{A+B}(z,z^*)&\geq \sup_{\{n\in\NN\}}\left[\langle z_n,z^*\rangle+\tfrac{1}{4}n\alpha
+\langle z-z_n, v_n^*\rangle \right]\\
&\geq \sup_{\{n\in\NN\}}\left[-\| z_n\|\cdot\|z^*\|+\tfrac{1}{4}n\alpha
+\langle z-z_n, v_n^*\rangle \right]\\
&\geq \limsup_{n\rightarrow\infty}\left[-\| z_n\|\cdot\|z^*\|+\tfrac{1}{4}n\alpha
+\langle z-z_n, v_n^*\rangle \right]\\
&=\infty\quad\text{(by \eqref{FCG:1} and \eqref{FCG:2})}.
\end{align*}
Hence
\begin{align}
F_{A+B}(z,z^*)=\infty.
\end{align}
Therefore, we have verified  \eqref{FCGG:5} in all cases.
However,  \eqref{FCGG:5} contradicts our original choice
that $z\in P_X\left[\dom F_{A+B}\right]$.
Hence
$P_X\left[\dom F_{A+B}\right]\subseteq\overline{\dom A\cap\dom B}$ and thus \eqref{FCG:10} holds.
Thus $\overline{P_X\left[\dom F_{A+B}\right]}=\overline{\dom A\cap\dom B}$.
\end{proof}

\begin{corollary}\label{PGV:01}
Let $A:X\To X^*$ be maximal monotone of type  (FPV) with convex domain,
and $B: X\To X^*$ be maximal monotone
with $\dom A\cap\inte\dom B\neq\varnothing$.
Assume that $\dom A\cap\overline{\dom B}\subseteq\dom B$. Then
$\overline{P_X\left[\dom F_{A+B}\right]}=\overline{\dom A\cap\dom B}$.

\end{corollary}
\begin{proof}
Combine Fact~\ref{f:referee02c}, Corollary~\ref{domain:L1} and Proposition~\ref{PGV:1}.
\end{proof}

\begin{corollary}\label{PGV:02}
Let $A:X\To X^*$ be a maximal monotone linear relation,
and let $B: X\To X^*$ be maximal monotone
with $\dom A\cap\inte\dom B\neq\varnothing$.
Assume that $\dom A\cap\overline{\dom B}\subseteq\dom B$.
Then $\overline{P_X\left[\dom F_{A+B}\right]}=\overline{\dom A\cap\dom B}$.

\end{corollary}
\begin{proof}
Combine
Fact~\ref{f:referee02c}, Corollary~\ref{domain:L2} and
Proposition~\ref{PGV:1}.
Alternatively, combine Fact~\ref{f:referee01} and Corollary~\ref{PGV:01}.
\end{proof}

We are now ready for our main result.

\begin{theorem}[Main Result] \label{t:main}
Let $A, B:X\To X^*$ be  maximal monotone
with $\dom A\cap\inte\dom B\neq\varnothing$. Assume that $A+N_{\overline{\dom B}}$
 is maximal monotone of type (FPV), and that
$\dom A\cap\overline{\dom B}\subseteq\dom B$.
Then $A+B$ is maximal monotone.
\end{theorem}

\begin{proof}
After translating the graphs if necessary, we can and do assume that
$0\in\dom A\cap\inte\dom B$ and that $(0,0)\in\gra A\cap\gra B$.
By Fact~\ref{f:Fitz}, $\dom A\subseteq P_X(\dom F_A)$ and $\dom B\subseteq P_X(\dom F_{B})$.
Hence,
\begin{align}\bigcup_{\lambda>0} \lambda
\big(P_X(\dom F_A)-P_X(\dom F_{B})\big)=X.\end{align}
Thus, by Fact~\ref{f:referee1}, it suffices to show that
\begin{equation} \label{e0:ourgoal}
F_{A+ B}(z,z^*)\geq \langle z,z^*\rangle,\quad \forall(z,z^*)\in X\times X^*.
\end{equation}
Take $(z,z^*)\in X\times X^*$.
Then
\begin{align}
&F_{A+B}(z,z^*)\nonumber\\
&=\sup_{\{x,x^*,y^*\}}\left[\langle x,z^*\rangle+\langle z,x^*\rangle-\langle x,x^*\rangle
+\langle z-x, y^*\rangle -\iota_{\gra A}(x,x^*)-\iota_{\gra
B}(x,y^*)\right].\label{see:1}
\end{align}
Assume to the contrary  that
\begin{align}
F_{A+B}(z,z^*)<\langle z,z^*\rangle.\label{See:1a4}
\end{align}
Then $(z,z^*)\in\dom F_{A+B}$ and,
by Proposition~\ref{PGV:1},
 \begin{align}z\in\overline{\dom A\cap\dom B}
=\overline{P_X\left[\dom F_{A+B}\right]}.\label{SeF:1}
\end{align}
Next, we show that
\begin{align}
F_{A+B}(\lambda z,\lambda z^*)\geq \lambda^2\langle z,z^*\rangle,
\quad\forall \lambda\in\left]0,1\right[.\label{See:10}\end{align}
Let $\lambda\in\left]0,1\right[$.
By \eqref{SeF:1} and Fact~\ref{f:referee02c}, $z\in \overline{P_X\dom F_B}$. By Fact~\ref{f:referee02c} again and $0\in\inte\dom B$,
$0\in\inte\overline{P_X\dom F_B}$.
Then, by \cite[Theorem~1.1.2(ii)]{Zalinescu}, we have
\begin{align}
\lambda z\in\inte\overline{P_X\dom F_B}=\inte\left[P_X\dom
F_B\right]\label{ReAu:1}.
\end{align}
Combining \eqref{ReAu:1} and Fact~\ref{f:referee02c},  we see that
$\lambda z\in\inte\dom B$.

We consider two cases.

\emph{Case 1}: $\lambda z\in\dom A$.\\
By \eqref{see:1},
\begin{align*}
&F_{A+B}(\lambda z,\lambda z^*)\nonumber\\
&\geq\sup_{\{x^*,y^*\}}\left[\langle \lambda z,\lambda z^*\rangle+\langle \lambda z,x^*\rangle-\langle \lambda z,x^*\rangle
+\langle \lambda z-\lambda z, y^*\rangle -\iota_{\gra A}(\lambda z,x^*)-\iota_{\gra B}(\lambda z,y^*)\right]\\
&=\langle \lambda z,\lambda z^*\rangle.
\end{align*}
Hence \eqref{See:10} holds.

\emph{Case 2}: $\lambda z\notin\dom A$.\\
Using $0\in\dom A\cap\dom B$ and
the convexity of $\overline{\dom A\cap \dom B}$ (which follows from
\eqref{SeF:1}), we obtain
$\lambda z\in\overline{\dom A\cap \dom B}\subseteq \overline{\dom A\cap \overline{\dom B}}$.
Set
 \begin{align}
U_n= \lambda z+\tfrac{1}{n}\mathbb{B}_X,\quad \forall n\in\NN\label{FD:1}.
\end{align}
Then $U_n\cap \dom (A+N_{\overline{\dom B}})\neq\varnothing$.
Since $(\lambda z, \lambda z^*)\notin\gra (A+N_{\overline{\dom B}})$,
$\lambda z\in U_n$, and $A+N_{\overline{\dom B}}$ is of type (FPV),
 there exists $(b_n,b^*_n)\in\gra (A+N_{\overline{\dom B}})$
such that $b_n\in U_n$  and
\begin{align}
\langle \lambda z,b^*_n\rangle+\langle b_n,\lambda z^*\rangle-
\langle b_n,b^*_n\rangle>\lambda^2\langle z,z^*\rangle,\quad \forall n\in\NN.\label{see:20}
\end{align}
Since $\lambda z\in\inte\dom B$ and $b_n\rightarrow \lambda z$, by Fact~\ref{pheps:11},
 there exist $N\in\NN$ and  $M>0$ such that
\begin{align}b_n\in\inte\dom B \quad\text{and}\quad \sup_{v^*\in Bb_n}\| v^*\|
\leq M,\quad \forall n\geq N.\label{See:1a2}
\end{align}
Hence
 $N_{\overline{\dom B}}(b_n)=\{0\}$ and thus $(b_n,b^*_n)\in\gra A$ for every $n\geq N$.
Thus by \eqref{see:1}, \eqref{see:20} and \eqref{See:1a2},
\begin{align}
&F_{A+B}(\lambda z,\lambda z^*)\nonumber\\
&\geq\sup_{\{ v^*\in Bb_n\}}\left[\langle b_n,\lambda z^*\rangle+\langle \lambda z,b_n^*\rangle-\langle b_n,b_n^*\rangle
+\langle \lambda z-b_n, v^*\rangle \right],\quad \forall n\geq N\nonumber\\
&\geq\sup_{\{ v^*\in Bb_n\}}\left[\lambda^2\langle z,z^*\rangle
+\langle \lambda z-b_n,v^*\rangle \right],\quad \forall n\geq N\quad\text{(by \eqref{see:20})}\nonumber\\
&\geq\sup\left[\lambda^2\langle z,z^*\rangle
-M\|\lambda z-b_n\| \right],\quad \forall n\geq N\quad\text{(by \eqref{See:1a2})}\nonumber\\
&\geq \lambda^2\langle z,z^*\rangle\quad\text{(by $b_n\rightarrow \lambda z$)}.\label{See:1a3}
\end{align}
Hence $
F_{A+B}(\lambda z,\lambda z^*)\geq \lambda^2\langle z,z^*\rangle$.

We have verified that \eqref{See:10} holds in both cases.
Since $(0,0)\in\gra A\cap\gra B$, we obtain
$(\forall (x,x^*)\in\gra (A+B))$ $\langle x,x^*\rangle\geq0$.
Thus, $F_{A+B}(0,0)=0$.
Now define
\begin{align*}
f\colon \left[0,1\right]\rightarrow\RR \colon  t\rightarrow F_{A+B}(tz,tz^*).
\end{align*}
Then $f$ is continuous on $\left[0,1\right]$ by \cite[Proposition~2.1.6]{Zalinescu}.
From \eqref{See:10}, we obtain
\begin{align}
F_{A+B}(z,z^*)=\lim_{\lambda\rightarrow1^{-}}
F_{A+B}(\lambda z,\lambda z^*)\geq \lim_{\lambda\rightarrow1^{-}}
\langle \lambda z,\lambda z^*\rangle=\langle z,z^*\rangle,
\end{align}
which contradicts \eqref{See:1a4}. Hence \begin{align}
F_{A+B}(z,z^*)\geq \langle z,z^*\rangle.
\end{align}
Therefore,
 \eqref{e0:ourgoal} holds, and $A+B$ is maximal monotone.
\end{proof}

Theorem~\ref{t:main} allows us to deduce both new and previously known sum
theorems.

\begin{corollary}\label{Veron:1}
Let $f:X\rightarrow\RX$ be proper, lower semicontinuous, convex,
and let $ B:X\To X^*$ be  maximal monotone
with $\dom f\cap\inte\dom B\neq\varnothing$.
Assume that $\dom \partial f\cap\overline{\dom B}\subseteq\dom B$.
Then $\partial f+B$ is maximal monotone.
\end{corollary}
\begin{proof}
By Fact~\ref{f:referee02c} and Fact~\ref{f:F4}, $\partial f+N_{\overline{\dom B}}=\partial (f+\iota_{\overline{\dom B}})$.
Then by Fact~\ref{f:referee0d}, $\partial f+N_{\overline{\dom B}}$ is type of (FPV).
Now apply Theorem~\ref{t:main}.
\end{proof}

\begin{corollary}\label{Coro:001}
Let $A:X\To X^*$ be maximal monotone of type  (FPV),
and let $B:X\To X^*$ be  maximal monotone  with full domain.
Then $A+B$ is maximal monotone.
\end{corollary}
\begin{proof} Since $A+N_{\overline{\dom B}}=A+N_X=A$ and thus $A+N_{\overline{\dom B}}$ is maximal
monotone of type (FPV),
the conclusion follows from  Theorem~\ref{t:main}.
\end{proof}

\begin{corollary}[Verona-Verona]
\emph{(See \cite[Corollary~2.9(a)]{VV} or \cite[Theorem~53.1]{Si2}.)}
\label{Veron:2}
Let $f:X\rightarrow\RX$ be proper, lower semicontinuous, and  convex,
and let $ B:X\To X^*$ be  maximal monotone with full domain.
Then $\partial f+B$ is maximal monotone.
\end{corollary}
\begin{proof}
Clear from Corollary~\ref{Veron:1}.
Alternatively, combine Fact~\ref{f:referee0d} and Corollary~\ref{Coro:001}.
\end{proof}

\begin{corollary}[Heisler]
\label{Corh:001}
\emph{(See \cite[Remark, page~17]{ph2}.)}
Let $A,B:X\To X^*$ be  maximal monotone with full domain.
Then $A+B$ is maximal monotone.
\end{corollary}
\begin{proof} Let $C$ be a nonempty closed convex subset of $X$.
 By Corollary~\ref{Veron:2}, $N_C+A$ is maximal monotone.
Thus, $A$ is of type (FPV)  by Fact~\ref{f:referee02a}.
The conclusion now follows from Corollary~\ref{Coro:001}.
\end{proof}

\begin{corollary}
\label{Coro:00a1}
Let $A:X\To X^*$ be maximal monotone of type  (FPV)  with convex domain,
and let $B:X\To X^*$ be maximal monotone
with $\dom A\cap\inte\dom B\neq\varnothing$.
Assume that $\dom A\cap\overline{\dom B}\subseteq\dom B$.
Then $A+B$ is maximal monotone.
\end{corollary}
\begin{proof}
Combine Fact~\ref{f:referee02c}, Corollary~\ref{domain:L1} and Theorem~\ref{t:main}.
\end{proof}

\begin{corollary}[Voisei]
\emph{(See \cite{Voisei09}.)}
\label{Coro:01}
Let $A:X\To X^*$ be maximal monotone of type  (FPV)  with convex domain,
let $C$ be a nonempty closed convex subset of $X$,
and suppose that $\dom A \cap \inte C\neq \varnothing$.
Then $A+N_C$ is maximal monotone.
\end{corollary}
\begin{proof}
Apply Corollary~\ref{Coro:00a1}.
\end{proof}

\begin{corollary}\label{Coro:002de}
Let $A:X\To X^*$ be a maximal monotone linear relation,
and let $B:X\To X^*$ be maximal monotone
with $\dom A\cap\inte\dom B\neq\varnothing$.
Assume that $\dom A\cap\overline{\dom B}\subseteq\dom B$.
Then $A+B$ is maximal monotone.
\end{corollary}

\begin{proof}
Combine Fact~\ref{f:referee01} and Corollary~\ref{Coro:00a1}.
\end{proof}

\begin{corollary}\emph{(See \cite[Theorem~3.1]{BWY4}.)}\label{Coro:1}
Let $A:X\To X^*$ be a maximal monotone linear relation,
let $C$ be a nonempty closed convex subset of $X$,
and suppose that $\dom A \cap \inte C\neq \varnothing$.
Then $A+N_C$ is maximal monotone.
\end{corollary}
\begin{proof}
Apply Corollary~\ref{Coro:002de}.
\end{proof}

\begin{corollary}\label{Coro:002}
Let $A:X\To X^*$ be a maximal monotone linear relation,
and let $B:X\To X^*$ be maximal monotone  with full domain.
Then $A+B$ is maximal monotone.
\end{corollary}

\begin{proof}
Apply Corollary~\ref{Coro:002de}.
\end{proof}

\begin{example}
\label{ex:main}
Suppose that $X=L^1[0,1]$, let $$D=\menge{x\in X}{\text{$x$ is absolutely
continuous}, x(0)=0, x'\in X^*},$$
and set
\begin{equation*}
A\colon X\To X^*\colon
x\mapsto \begin{cases}
\{x'\}, &\text{if $x\in D$;}\\
\varnothing, &\text{otherwise.}
\end{cases}
\end{equation*}
By Phelps and Simons' \cite[Example~4.3]{PheSim},
$A$ is an at most single-valued maximal monotone linear relation
with proper dense domain, and $A$ is neither symmetric nor skew.
Now let $J$ be the \emph{duality mapping},
i.e., $J=\partial\tfrac{1}{2}\|\cdot\|^2$.
Then Corollary~\ref{Coro:002} implies that $A+J$ is maximal monotone.
To the best of our knowledge, the maximal monotonicity of
$A+J$ cannot be deduced from
any previously known result.
\end{example}

\begin{remark}
In \cite{BWY9},
it was shown that
the sum theorem is true when $A$ is a linear relation, $B$ is the subdifferential operator
of a proper lower semicontinuous sublinear function, and Rockafellar's
constraint qualification holds.
When the domain of the subdifferential operator is closed, then that result
can be deduced from  Theorem~\ref{t:main}.
However, it is possible that the domain of the subdifferential operator
of a proper lower semicontinuous sublinear function does not have to  be closed.
For an example, see \cite[Example~5.4]{BBBRW}:
Set $C=\{(x,y)\in\RR^2\mid 0<1/x\leq y\}$ and $f=\iota^*_C$.
Then $f$ is not subdifferentiable at any point in the boundary of its domain,
except at the origin.
Thus, in the general case,
we do not know whether or not it is possible to deduce the result in
\cite{BWY9} from Theorem~\ref{t:main}.
\end{remark}

\section*{Acknowledgment}
The author thanks Dr.\ Heinz Bauschke and Dr.\ Xianfu Wang for their
valuable discussions and comments. The author also thanks
Dr.\ Robert Csetnek for his pertinent comments.

\end{document}